\documentclass{amsart}
\usepackage{verbatim}
\usepackage{graphicx}
\usepackage{amsmath}
\usepackage{amscd}
\usepackage{amsfonts}
\usepackage{amssymb}

\newcommand{\be}{\begin{equation}}
\newcommand{\ee}{\end{equation}}
\newcommand{\bes}{\begin{equation*}}
\newcommand{\ees}{\end{equation*}}

\newcommand{\cA}{\mathcal{A}}

\newcommand{\cH}{\mathcal{H}}

\newcommand{\cM}{\mathcal{M}}

\newcommand{\cL}{\mathcal{L}}

\newcommand{\cR}{\mathcal{R}}
\newcommand{\cS}{\mathcal{S}}
\newcommand{\tT}{\tilde{T}}

\newcommand{\Rp}{\mathbb{R}_+}
\newcommand{\Rpt}{\mathbb{R}_+^2}

\newcommand{\diad}{\mathbb{D}^2}
\newcommand{\diadp}{\mathbb{D}_+^2}
\newcommand{\dpp}{\mathbb{D}_{++}}
\newcommand{\diadpp}{\mathbb{D}_{++}^2}
\newcommand{\mb}[1]{\mathbb{#1}}

\begin{document}

\title[E-dilation of commuting CP-semigroups]
{E-dilation of strongly commuting CP-semigroups (the nonunital case)}
\author[Orr M. Shalit]
{Orr Moshe Shalit}

\address{Pure Mathematics Dept.\\
University of Waterloo\\
Waterloo, ON\; N2L--3G1\\
CANADA} \email{oshalit@math.uwaterloo.ca}

\thanks{The author was partly supported by the Gutwirth Fellowship.}

\keywords{CP-semigroup, E-semigroup, two-parameter, quantum Markov semigroup, dilation, product system, completely contractive representation, isometric dilation}

\subjclass[2000]{46L55, 46L57, 46L08}

\begin{abstract}
In a previous paper, we showed that every strongly commuting pair of CP$_0$-semigroups
on a von Neumann algebra (acting on a separable Hilbert space) has an E$_0$-dilation. In this paper we show that if one restricts attention
to the von Neumann algebra $B(H)$ then the unitality assumption can be dropped, that is, we prove that
every pair of strongly commuting CP-semigroups on $B(H)$ has an E-dilation. The proof is significantly different from the proof for the unital case,
and is based on a construction of Ptak from the 1980's designed originally for constructing
a unitary dilation to a two-parameter contraction semigroup.
\end{abstract}
 
\maketitle

%
\newtheorem{theorem}{Theorem}[section]
\newtheorem{lemma}[theorem]{Lemma}
\newtheorem{proposition}[theorem]{Proposition}
\newtheorem{conjecture}[theorem]{Conjecture}
\newtheorem{example}[theorem]{Example}
\newtheorem{defin}[theorem]{Definition}
\newtheorem{corollary}[theorem]{Corollary}
\newtheorem{remark}[theorem]{Remark}
\newtheorem{remarks}[theorem]{Remarks}
\newtheorem{notation}[theorem]{Notation}
\newtheorem{exremark}[theorem]{Extended Remark and Notation}
\newtheorem{conclremarks}[theorem]{Concluding Remarks}

\newtheorem*{Prov}{Provision}
\newtheorem {mythm}{Theorem}[section]
\newtheorem {mylem}{Lemma}
\newtheorem*{idfn}{Definition}
\newenvironment{definition}{\begin{idfn}
\rm}{\end{idfn}}

\section{Introduction}
Let $H$ be a separable Hilbert space, and let $\cM \subseteq B(H)$ be a von Neumann algebra. A \emph{CP-semigroup} on $\cM$ is a family $P = \{P_t\}_{t\geq0}$ of contractive normal completely positive maps on $\cM$ satisfying the semigroup property
$$P_{s+t}(a) = P_s (P_t(a)) \,\, ,\,\, s,t\geq 0, a\in \cM ,$$
$$P_{0}(a) = a \,\, , \,\,  a\in \cM ,$$
and the continuity condition
$$\lim_{t\rightarrow t_0} \langle P_t(a)h,g\rangle = \langle P_{t_0}(a)h,g\rangle \,\, , \,\, a\in \cM, h,g \in H .$$ 
A CP-semigroup is called an \emph{E-semigroup} if each of
its elements is a $*$-endomorphism. A CP-semigroup is called \emph{unital} if for all $t$, $P_t$ is a unit preserving map. Unital CP-semigroups (E-semigroups) are called CP$_0$-semigroups (E$_0$-semigroups). CP$_0$-semigroups are also referred to sometimes as \emph{Markov semigroups}. For a thorough exposition on E-semigroups, containing a serious discussion on CP-semigroups,
see \cite{Arv03}.

Given a CP-semigroup $P$ we say that a
quadruple $(K,u,\cR,\alpha)$ is an \emph{E-dilation} of
$P$ if $K$ is a Hilbert space, $u: H \rightarrow K$ is an
isometry, $\cR$ is a von Neumann algebra satisfying $u^* \cR u =
\cM$, and $\alpha$ is an E-semigroup such that
\be\label{eq:dilation}
P_t(u^* b u) = u^* \alpha_t (b) u \,\, , \,\, b \in \cR
\ee
for all
$t \geq 0$. It has been proved by several authors that every CP-semigroup has an
E-dilation: Bhat \cite{Bhat96} (see also \cite{Bhat99}), SeLegue \cite{SeLegue}, Bhat--Skeide \cite{BS00}, Muhly--Solel \cite{MS02} and Arveson \cite{Arv03} (some of the authors require that $P$ be unital, or that $\cM = B(H)$).

In \cite{ShalitCP0Dil} we raised the question whether every \emph{two-parameter} CP$_0$-semigroup
has a (two-parameter) E$_0$-dilation. We showed that
given a pair $R = \{R_t\}_{t\geq 0}$ and $S = \{S_t\}_{t\geq 0}$ of \emph{strongly commuting} CP$_0$-semigroups
on a von Neumann algebra $\cM \subseteq B(H)$, there exists a dilation $(K,u,\cR,\alpha)$ (minimal in a certain sense), with $K$,$u$, and $\cR$ as in the previous paragraph and $\alpha = \{\alpha_t\}_{t \in \Rpt}$ a two-parameter E$_0$-semigroup satisfying (\ref{eq:dilation}) for all $t\in\Rpt$, where $P$ is the two-parameter semigroup given by
$$P_{(s,t)} = R_s \circ S_t .$$
We postpone the definition of strong commutativity to Section \ref{sec:preliminaries} below, but we mention straight away as an example that if $R = \{R_t\}_{t\geq 0}$ is a CP-semigroup, $S = \{S_t\}_{t\geq 0}$ is a semigroup of normal $*$-automorphisms and $R$ and $S$ commute, then $R$ and $S$ strongly commute.

In this paper we try to drop the unitality assumption on $R$ and $S$. In order to do so, we have to restrict attention to the von Neumann algebra $\cM = B(H)$.
It is interesting that dropping the assumption on unitality forces us to completely change approximately one-half of the proof, and we have to introduce new methods for both the algebraic and the analytic parts. Much of the results in this paper hold for semigroups acting on von Neumann algebras more general than $B(H)$, but the bottom line -- the existence of an E-dilation -- is proved only for $\cM = B(H)$.

It is worth noting that in \cite{Bhat98}, Bhat showed that given a pair of commuting CP maps $\Theta$ and $\Phi$ on $B(H)$, there is a Hilbert space $K\supseteq H$ and a pair of commuting normal $*$-emdomorphisms $\alpha$ and $\beta$ acting on $B(K)$ such that
\bes
\Theta^m \circ \Phi^n (u^* b u) = u^* \alpha^m \circ \beta^n (b) u \,\, , \,\, b \in B(K)
\ees
for all $m,n \in \mathbb{N}$ (here $u$ denotes the inclusion of $H$ onto $K$). Later on Solel, using a different method, proved this result for commuting CP maps on arbitrary von Neumann algebras \cite{S06a}. Neither one of the above results requires strong commutativity. 

{\bf Overview of the paper.} In Section \ref{sec:preliminaries} we give some preliminary material in product systems of W$^*$-correspondences, representations of product systems, E-dilations of CP-semigroups and strong commutation. We also prove a new chararcterization of strong commutation in terms of the GNS representations.

Section \ref{sec:repSC} is a review of our constructions from \cite{ShalitCP0Dil} (which are, in turn, based on the constructions of Muhly and Solel from \cite{MS02}), which allow one to represent a pair of strongly commuting CP-semigroups via a product system representation in the form of
\bes
R_s\circ S_t (a) = \widetilde{T}_{(s,t)} \left(I_{X({(s,t)})}
\otimes a \right) \widetilde{T}^*_{(s,t)} ,
\ees
where $X = \{X({(s,t)})\}_{{(s,t)}\in\Rpt}$ is a product system of W$^*$-correspondences and $T$ is completely contractive covariant representation of $X$ on $H$.

In Section \ref{sec:iso_dil} we develop a method invented by Ptak to obtain an
isometric dilation of a product system representation over
$$\diadp := \left\{\left(\frac{k}{2^{n}},\frac{m}{2^{n}}\right) : k,m,n = 0,1,2,\ldots \right\} $$
(Ptak used this method to prove that every two-parameter semigroup of contractions on a separable Hilbert space has a unitary dilation, see \cite{Ptak}). The results of this section rely also on two papers of B. Solel, \cite{S06a} and \cite{S06b}, containing results on dilations of product system representations over $\mathbb{N}^2$.

Section \ref{sec:extension} is devoted to the problem of extending a CP-semigroup $\phi = \{\phi_s\}_{s\in\cS}$ parameterized by a dense subsemigroup $\cS \subseteq \Rp$ to a CP-semigroup $\hat{\phi} = \{\hat{\phi_t}\}_{t\geq 0}$ parameterized by $\Rp$. Using some ideas of SeLegue \cite{SeLegue} and Arveson \cite{Arv03}, we show that for a certain class of von Neumann algebras, including $B(H)$, such an extension always exists.

The construction of an E-dilation to a given CP-semigroup is done in Section \ref{sec:Edilation}, by using the results from Sections \ref{sec:repSC}--\ref{sec:extension}. It is noteworthy that, due to the non-unitality of the CP-semigroup,
this part of the proof is considerably trickier than the analogous part in the proof of existence of dilations for
unital CP-semigroups. As a corollary we prove that every pair of commuting CP-semigroups, one of which is a semigroup of $*$-automorphisms, has a minimal E-dilation.

\section{Preliminaries}\label{sec:preliminaries}

\subsection{$C^*$/$W^*$-correspondences, their products and their representations}\label{subsec:correspondences}
\begin{defin}
Let $\cA$ be a $C^*$-algebra. A \emph{Hilbert
$C^*$-correspondences over $\cA$} is a (right) Hilbert
$\cA$-module $E$ which carries an adjointable, left action of
$\cA$.
\end{defin}
\begin{defin}
Let $\cM$ be a $W^*$-algebra. A \emph{Hilbert
$W^*$-correspondence over $\cM$} is a self-adjoint Hilbert
$C^*$-correspondence $E$ over $\cM$, such that the canonical map from $\cM$ into the W$^*$-algebra $\cL(E)$ (which gives rise to the left action) is normal.
\end{defin}
\begin{defin}
Let $E$ be a $C^*$-correspondence over $\cA$, and let $H$ be a
Hilbert space. A pair $(\sigma, T)$ is called a \emph{completely
contractive covariant representation} of $E$ on $H$ (or, for
brevity, a \emph{c.c. representation}) if
\begin{enumerate}
    \item $T: E \rightarrow B(H)$ is a completely contractive linear map;
    \item $\sigma : A \rightarrow B(H)$ is a nondegenerate $*$-homomorphism; and
    \item $T(xa) = T(x) \sigma(a)$ and $T(a\cdot x) = \sigma(a) T(x)$ for all $x \in E$ and  all $a\in\cA$.
\end{enumerate}
If $\cA$ is a $W^*$-algebra and $E$ is $W^*$-correspondence then
we also require that $\sigma$ be normal.
\end{defin}
Given a $C^*$-correspondence $E$ and a c.c. representation
$(\sigma,T)$ of $E$ on $H$, one can form the Hilbert space $E
\otimes_\sigma H$, which is defined as the Hausdorff completion of
the algebraic tensor product with respect to the inner product
$$\langle x \otimes h, y \otimes g \rangle = \langle h, \sigma (\langle x,y\rangle) g \rangle .$$
One then defines $\widetilde{T} : E \otimes_\sigma H \rightarrow H$ by
$$\widetilde{T} (x \otimes h) = T(x)h .$$

\begin{defin}
A c.c. representation $(\sigma, T)$ is called \emph{isometric} if
for all $x, y \in E$,
\begin{equation*}
T(x)^*T(y) = \sigma(\langle x, y \rangle) .
\end{equation*}
(This is the case if and only if $\widetilde{T}$ is an isometry.) It
is called \emph{fully coisometric} if $\widetilde{T}$ is a coisometry.
\end{defin}

Given two Hilbert $C^*$-correspondences $E$ and $F$ over $\cA$,
the \emph{balanced} (or \emph{inner}) tensor product $E
\otimes F$ is a Hilbert $C^*$-correspondence over $\cA$
defined to be the Hausdorff completion of the algebraic tensor
product with respect to the inner product
$$\langle x \otimes y, w \otimes z \rangle = \langle y , \langle x,w\rangle \cdot z \rangle \,\, , \,\,  x,w\in E, y,z\in F .$$
The left and right actions are defined as $a \cdot (x \otimes y) =
(a\cdot x) \otimes y$ and $(x \otimes y)a = x \otimes (ya)$,
respectively, for all $a\in A, x\in E, y\in F$. When working
in the context of $W^*$-correspondences, that is, if $E$ and $F$
are $W$*-correspondences and $\cA$ is a $W^*$-algebra, then $E
\otimes F$ is understood to be the (unique minimal) \emph{self-dual
extension} of the above construction (see \cite{Pas}).

Suppose $\cS$ is an abelian cancellative semigroup with identity
$0$ and $p : X \rightarrow \cS$ is a family of
$W^*$-correspondences over $\cA$. Write $X(s)$ for the
correspondence $p^{-1}(s)$ for $s \in \cS$. We say that $X$ is a
\emph{product system} over $\cS$ if $X$ is a semigroup,
$p$ is a semigroup homomorphism and, for each $s,t \in \cS
\setminus \{0\}$, the map $X(s) \times X(t) \ni (x,y) \mapsto xy
\in X(s+t)$ extends to an isomorphism $U_{s,t}$ of correspondences
from $X(s) \otimes X(t)$ onto $X(s+t)$. The associativity of
the multiplication means that, for every $s,t,r \in \cS$,
\begin{equation}\label{eq:assoc_prod}
U_{s+t,r} \left(U_{s,t} \otimes I_{X(r)} \right) = U_{s,t+r} \left(I_{X(s)} \otimes U_{t,r} \right).
\end{equation}
We also require that $X(0) = \cA$ and that the multiplications
$X(0) \times X(s) \rightarrow X(s)$ and $X(s) \times X(0)
\rightarrow X(s)$ are given by the left and right actions of $\cA$
and $X(s)$.

\begin{defin}
Let $H$ be a Hilbert space, $\cA$ a $W^*$-algebra and $X$ a
product system of Hilbert $\cA$-correspondences over the semigroup
$\cS$. Assume that $T : X \rightarrow B(H)$, and write $T_s$ for
the restriction of $T$ to $X(s)$, $s \in \cS$, and $\sigma$ for
$T_0$. $T$ (or $(\sigma, T)$) is said to be a \emph{completely
contractive covariant representation} of $X$ if
\begin{enumerate}
    \item For each $s \in \cS$, $(\sigma, T_s)$ is a c.c. representation of $X(s)$; and
    \item $T(xy) = T(x)T(y)$ for all $x, y \in X$.
\end{enumerate}
$T$ is said to be an isometric (fully coisometric) representation if it is an isometric (fully coisometric) representation on every fiber $X(s)$.
\end{defin}
Since we shall not be concerned with any other kind of representation, we shall call a completely contractive covariant representation of a product system simply a \emph{representation}.

\subsection{CP-semigroups and E-dilations}

Let $\cS$ be a unital subsemigroup of $\Rp^k$, and let $\cM$ be a
von Neumann algebra acting on a Hilbert space $H$. A \emph{CP map}
is a completely positive, contractive and normal map on $\cM$. A
\emph{CP-semigroup over $\cS$} is a family $\{P_s\}_{s \in
\cS}$ of CP maps on $\cM$ such that
\begin{enumerate}
    \item For all $s,t \in \cS$
    $$P_s \circ P_t = P_{s + t} \,;$$
    \item $P_0 = {\bf id}_{\cM}$;
    \item For all $h,g\in H$ and all $a\in \cM$, the function
    $$\cS\ni s \mapsto \langle P_s(a) h,g \rangle $$
    is continuous.
\end{enumerate}
A CP-semigroup is called an \emph{E-semigroup} if it consists of $*$-endomorphisms. A CP (E) - semigroup is
called a \emph{CP$_0$ (E$_0$)-semigroup} if all its elements are unital.

\begin{defin}\label{def:dilation}
Let $\cM$ be a von Neumann algebra of operators acting on a
Hilbert space $H$, and let $P = \{P_s\}_{s \in \cS}$ be
a CP-semigroup over the semigroup $\cS$. An \emph{E-dilation of
$P$} (or, simply, a \emph{dilation of $P$}) is a quadruple $(K,u,\cR,\alpha)$, where $K$ is a
Hilbert space, $u: H \rightarrow K$ is an isometry, $\cR$ is a von
Neumann algebra satisfying $u^* \cR u = \cM$, and $\alpha$ is an
E-semigroup over $\cS$ such that \be\label{eq:CPdef_dil1} P_s
(u^* a u) = u^* \alpha_s (a) u \,\, , \,\, a \in \cR \ee for all
$s \in \cS$.

If $(K,u,\cR,\alpha)$ is a dilation of $P$, then $(\cM,
P)$ is called a \emph{compression} of $(K,u,\cR,\alpha)$.
\end{defin}

If one identifies $\cM$ with $u \cM u^*$, $H$ with $u H$,
and defines $p := u u^*$,
one may give the following equivalent definition,
which we shall use interchangeably with definition
\ref{def:dilation}: \emph{a triple $(p,\cR,\alpha)$ is called a
\emph{dilation} of $P$ if $\cR$ is a von Neumann algebra
containing $\cM$, $\alpha$ is an E-semigroup on $\cR$ and $p$ is a
projection in $\cR$ such that $\cM = p \cR p$ and
$$P_s (p a p) = p \alpha_s (a) p $$
holds for all $s \in \cS, a \in \cR$.}

With this change of notation, one easily sees that for all $s \in \cS$, $\alpha_s (1-p) \leq 1 - p$.
A projection with this property is called \emph{coinvariant} (for $\alpha$).

\begin{defin}\label{def:min_dil}
Let $(K,u,\cR,\alpha)$ be an E-dilation of the CP-semigroup
$P$.  $(K,u,\cR,\alpha)$ is said to be a \emph{minimal} dilation
if there is no multiplicative, coinvariant projection $1 \neq q
\in \cR$ such that $u u^* \leq q$, and if
\be\label{eq:W^*-generator} 
\cR = W^*\left(\bigcup_{s \in \cS}
\alpha_s (u \cM u^*)\right) . 
\ee
\end{defin}

In \cite{Arv03} Arveson defines a minimal dilation (for a CP-semigroup over $\Rp$) slightly
differently:
\begin{defin}\label{def:min_dil_Arv}
Let $(K,u,\cR,\alpha)$ be an E-dilation of the CP-semigroup
$P$.  $(K,u,\cR,\alpha)$ is said to a \emph{minimal} dilation
if the central support of $u u^*$ in $\cR$ is $1$, and if
(\ref{eq:W^*-generator}) holds.
\end{defin}
The two definitions have been shown to be equivalent in the case
where $P$ is a CP$_0$-semigroup over $\Rp$ (\cite{Arv03}, Section 8.9). The following proposition appeared in \cite[Subsection 2.2]{ShalitCP0Dil}, where a longer discussion of minimality is presented.

\begin{proposition}\label{prop:equiv_def_min}
Definition \ref{def:min_dil} holds if \ref{def:min_dil_Arv} does.
\end{proposition}

\subsection{Strong commutativity}

Let $\Theta$ and $\Phi$ be CP maps on a von Neumann algebra $\cM \subseteq B(H)$. We define the Hilbert
space $\cM \otimes_\Phi \cM \otimes_\Theta H$ to be the Hausdorff
completion of the algebraic tensor product $\cM
\otimes_\textrm{alg} \cM \otimes_\textrm{alg} H$ with respect to
the inner product
$$\langle a \otimes b \otimes h, c \otimes d \otimes k \rangle = \langle h, \Theta(b^* \Phi(a^* c) d) k \rangle .$$
\begin{defin}\label{def:SC}
Let $\Theta$ and $\Phi$ be CP maps on $\cM$. We say that they
\emph{commute strongly} if there is a unitary $u: \cM \otimes_\Phi
\cM \otimes_\Theta H \rightarrow \cM \otimes_\Theta \cM
\otimes_\Phi H$ such that:
\begin{itemize}
\item[(i)] $u(a \otimes_\Phi I \otimes_\Theta h) = a \otimes_\Theta I \otimes_\Phi h$ for all $a \in \cM$ and $h \in H$.
\item[(ii)] $u(ca\otimes_\Phi b \otimes_\Theta h) = (c \otimes I_M \otimes I_H)u(a\otimes_\Phi b \otimes_\Theta h)$ for $a,b,c \in \cM$ and $h \in H$.
\item[(iii)] $u(a\otimes_\Phi b \otimes_\Theta dh) = (I_M \otimes I_M \otimes d)u(a\otimes_\Phi b \otimes_\Theta h)$ for $a,b \in \cM$, $d \in \cM'$ and $h \in H$.
\end{itemize}
\end{defin}
The notion of strong commutation was introduced by Solel in
\cite{S06a}. Note that if two CP maps commute strongly, then they
commute. The appendix of \cite{ShalitCP0Dil} contains many examples of strongly commuting CP maps, and also particular necessary and sufficient conditions for strong commutativity to hold in particular von Neumann algebras. Let us recall the following characterization due to Solel of strongly commuting CP maps on $\cM = B(H)$.

Let $\Theta$ and $\Phi$ be two
CP maps on $B(H)$. It is well known that there are two
$\ell^2$-independent row contractions $\{T_{i}\}_{i=1}^{m}$ and
$\{S_{j}\}_{j=1}^{n}$ ($m,n$ may be equal to $\infty$) such that for all $a \in B(H)$
\be\label{eq:conj1}
\Theta(a) = \sum_{i}T_{i}aT_{i}^* ,
\ee
and
\be\label{eq:conj2}
\Phi(a) = \sum_{j}S_{j}aS_{j}^* \,.
\ee

\begin{theorem}\label{thm:Solel} (\cite[Proposition 5.8]{S06a})
Let $\Theta$ and $\Phi$ be CP maps on $B(H)$ given by (\ref{eq:conj1}) and (\ref{eq:conj2}) with
$\{T_{i}\}_{i=1}^{m}$ and
$\{S_{j}\}_{j=1}^{n}$ $\ell^2$-independent row contractions.
$\Theta$ and $\Phi$ commute strongly if and only if there is an $mn\times mn$
unitary matrix
$$u = \left(u_{(i,j)}^{(k,l)}\right)_{(i,j),(k,l)}$$
such that for all $i,j$,
\be\label{eq:SCunitary}
T_{i}S_{j} =
\sum_{(k,l)}u_{(i,j)}^{(k,l)}S_{l}T_{k} .
\ee
\end{theorem}

As a simple example, if $\Phi$ and $\Psi$ are given by (\ref{eq:conj1}) and (\ref{eq:conj2}),
and $S_{j}$ commutes with $T_{i}$ for all $i,j$, then $\Phi$ and $\Psi$ strongly commute. The following
theorem provides us with many more examples of strong commutativity.

\begin{theorem}\label{thm:M_n(C)}
(\cite[Proposition 8.1]{ShalitCP0Dil}) If  $\,\dim H = n < \infty$ is a finite dimensional Hilbert space then any two CP maps on
$B(H) = M_n(\mathbb{C})$ that commute do so strongly.
\end{theorem}

\begin{example}
(A commuting pair of CP maps on $B(H)$ that do not commute strongly).
\emph{Let $\cM = B(\ell^2(\mathbb{N}))$, and identify every operator with its matrix representation with respect to the standard basis (here $\mathbb{N} = \{0,1,2, \ldots\}$). Let $\Theta$ be the map that takes a matrix to its diagonal, and let $\Phi$ be given by conjugation with the right shift. $\Theta$ is a (unital) CP map, $\Phi$ is a (non-unital) $*$-endomorphism, these two maps commute, but not strongly. Indeed, since $\Phi$ is an endomorphism, $\cM \otimes_\Theta \cM \otimes_\Phi H$ is spanned by simple tensors $A \otimes I \otimes h$, $A \in B(H), h \in H$. But note that
$$\|A \otimes I \otimes e_0\|^2 = \langle e_0, \Phi(\Theta(A^*A)) e_0 \rangle = 0 \,\,,$$
so this space is actually spanned by vectors of the form $A \otimes I \otimes h$, $A \in B(H), h \perp e_0$. Any such simple tensor is in the span of $\{E_{i,j} \otimes I \otimes e_{j+1} \}_{i,j \in \mathbb{N}}$, where $E_{i,j}$ denotes the usual matrix unit with $1$ in the $i,j$-th place, and zeros elsewhere.}

\emph{If $u : \cM \otimes_\Theta \cM \otimes_\Phi H \rightarrow \cM \otimes_\Phi \cM \otimes_\Theta H$ is a candidate for a unitary that satisfies the definition of strong commutation, then we must have
$$u(E_{i,j} \otimes_\Theta I \otimes_\Phi e_{j+1}) = E_{i,j} \otimes_\Phi I \otimes_\Theta e_{j+1} = I \otimes_\Phi E_{i+1,j+1} \otimes_\Theta e_{j+1} \,\,,$$
because $\Phi(E_{i,j})=E_{i+1,j+1}$.
On the other hand, the element $I \otimes_\Phi E_{1,0} \otimes_\Theta e_0$ is not zero:
$$\|I \otimes_\Phi E_{1,0} \otimes_\Theta e_0\|^2 = \langle e_0,\Theta(E_{0,1}\Phi(I) E_{1,0}) e_0 \rangle = \langle e_0, e_0 \rangle = 1\,\,,$$
while it is orthogonal to the image of $u$:}
\begin{align*}
\langle I \otimes_\Phi E_{1,0} \otimes_\Theta e_0, I \otimes_\Phi E_{i+1,j+1} \otimes_\Theta e_{j+1} \rangle
& = \langle e_0, \Theta(E_{0,1}\Phi(I)E_{i+1,j+1})e_0 \rangle \\
& = \langle e_0, \Theta(E_{0,1}E_{i+1,j+1})e_0 \rangle \\
& = \langle e_0, \Theta(\delta_{1,i+1}E_{0,j+1}) e_0 \rangle \\
& = \delta_{1,i+1} \delta_{0,j+1} = 0 \,\,.
\end{align*}
\emph{So $u$ is not onto, and cannot be a unitary. Thus $\Theta$ and $\Phi$ do not commute strongly.}
\end{example}

\subsection{Strong commutativity in terms of the GNS representation}

We now characterize strong commutation using the GNS representation.
This characterization may be interesting for two reasons. First, it shows that
the notion of strong-commutativity is representation free. Second, it provides the connection
between the work of Bhat and Skeide on dilation of one-parameter CP-semigroups \cite{BS00}, and
the constructions made in this paper and in \cite{ShalitCP0Dil}.

Let $\Theta$, $\Phi$ and $\cM$ be as in the previous subsection. We will denote by $(E,\xi)$ and $(F,\eta)$ the GNS representations of $\Theta$ and $\Phi$, respectively. That is $E = \cM \otimes_\Theta \cM$, the correspondence formed with the inner product
$$\langle a \otimes b, a' \otimes b' \rangle = b^* \Theta(a^* a') b',$$
$\xi = 1\otimes 1$ and one checks that $\langle \xi, a \xi \rangle = \Theta(a)$ for all $a \in \cM$. $(F , \eta)$ is defined
similarly. See section 2.14 in \cite{BS00}.

\begin{proposition}\label{prop:SCBS}
$\Theta$ and $\Phi$ commute strongly if and only if there exists a unitary (bimodule map)
$$w:E \otimes F \rightarrow F \otimes E$$
sending $\xi \otimes \eta$ to $\eta \otimes \xi$.
\end{proposition}
\begin{proof}
By \cite[Lemma 4.3]{ShalitCP0Dil} and the remarks right after, $\Theta$ and $\Phi$ strongly commute, i.e. there exists a unitary $u$
as in Definition \ref{def:SC}, if and only there exists an isomorphism of $\cM-\cM$-correspondences (a unitary bimodule map)
$$v : \cM \otimes_\Theta \cM \otimes_\Phi \cM \rightarrow \cM \otimes_\Phi \cM \otimes_\Theta \cM $$
such that
$$v(1 \otimes_\Theta 1 \otimes_\Phi 1) = 1 \otimes_\Phi 1 \otimes_\Theta 1 .$$
Here $\cM \otimes_\Theta \cM \otimes_\Phi \cM$ is the W$^*$-correspondence obtained from the tensor product using the inner product
$$\langle a \otimes b \otimes c, a' \otimes b' \otimes c'\rangle = c^* \Phi(b^*\Theta(a^* a')b')c',$$
with the obvious left and right actions of $\cM$\footnote{Proof: identify $\cM \otimes_\Theta \cM \otimes_\Phi H$ with $\cM \otimes_\Theta \cM \otimes_\Phi \cM \otimes_{\bf id} H$, and find $v$ such that $u = v \otimes I$.}.
$\cM \otimes_\Phi \cM \otimes_\Theta \cM$ is defined in the same way. But
$$\cM \otimes_\Theta \cM \otimes_\Phi \cM \cong \left(\cM \otimes_\Theta \cM \right) \otimes \left(\cM \otimes_\Phi \cM \right) = E \otimes F$$
as $W^*$-correspondences via the correspondence isomorphism
$$a \otimes_\Theta b \otimes_\Phi c \mapsto \left(a \otimes_\Theta b \right) \otimes \left(1 \otimes_\Phi c \right).$$
Indeed, this is a well defined isometry because it preserves inner products:
\begin{align*}
\langle \left(a \otimes_\Theta b \right) \otimes \left(1 \otimes_\Phi c \right), \left(a' \otimes_\Theta b' \right) \otimes \left(1 \otimes_\Phi c' \right) \rangle
&= \langle 1 \otimes_\Phi c , b^* \Theta(a^* a') b' \left(1 \otimes_\Phi c' \right)\rangle \\
&= c^* \Phi(b^* \Theta(a^* a') b')c' \\
&= \langle a \otimes b \otimes c, a' \otimes b' \otimes c'\rangle.
\end{align*}
It is onto because the tensor product $E \otimes F$ is balanced. It is clear that this map is a bimodule map. Moreover,
this maps sends $1 \otimes 1 \otimes 1$ to $\xi \otimes \eta$. Thus, the existence of an isomorphism $v$ as above is equivalent to the existence of a an isomorphism
$$w:E \otimes F \rightarrow F \otimes E$$
sending $\xi \otimes \eta$ to $\eta \otimes \xi$.
\end{proof}

\subsection{Strong commutativity of CP-semigroups}

Let $\{R_t\}_{t\geq0}$ and $\{S_t\}_{t\geq0}$ be two semigroups of CP maps such that $R_s$ strongly commutes with $S_t$ for all $s,t \geq 0$. 
This means that there is a family $\{v_{s,t}\}_{(s,t)\in\Rpt}$ of isomorphisms between $\cM$-correspondences $v_{s,t}: \cM \otimes_{R_s} \cM \otimes_{S_t} \cM \rightarrow \cM \otimes_{S_t} \cM \otimes_{R_s} \cM$ sending $I \otimes I \otimes I$ to $I \otimes I \otimes I$.

Fix $s,s',t \geq 0$. We define an isometry
\bes
\cM \otimes_{R_{s+s'}} \cM \otimes_{S_{t}} \cM \rightarrow  \cM \otimes_{R_{s}} \cM \otimes_{R_{s'}} \cM \otimes_{S_{t}} \cM
\ees
by 
\bes
a \otimes_{R_{s+s'}} b \otimes_{S_{t}} c \mapsto a \otimes_{R_{s}} I \otimes_{R_{s'}} b \otimes_{S_{t}} c.
\ees
We also define an isometry
\bes
\cM \otimes_{S_{t}} \cM \otimes_{R_{s+s'}} \cM \rightarrow  \cM \otimes_{S_{t}} \cM \otimes_{R_{s}} \cM \otimes_{R_{s'}} \cM
\ees
by 
\bes
a \otimes_{S_{t}} b \otimes_{R_{s+s'}} c \mapsto a \otimes_{S_{t}} b \otimes_{R_{s}} I \otimes_{R_{s'}} c.
\ees
We make similar definitions with the roles of $R$ and $S$ reversed.

\begin{defin}\label{def:SCSG}
Two semigroups of CP maps $\{R_t\}_{t\geq0}$ and
$\{S_t\}_{t\geq0}$ are said to \emph{commute strongly} if for all
$(s,t) \in \Rpt$ the CP maps $R_s$ and $S_t$ commute strongly, and if there is a family $\{v_{s,t}\}_{(s,t) \in \Rpt}$ of isomorphisms of $\cM$-correspondences $v_{s,t}: \cM \otimes_{R_s} \cM \otimes_{S_t} \cM \rightarrow \cM \otimes_{S_t} \cM \otimes_{R_s} \cM$ (making the $R_s$ and $S_t$ commute strongly) such that for all $s,s',t,t'\geq 0$ the following diagrams commute
\bes
\begin{CD}
\cM \otimes_{R_{s+s'}} \cM \otimes_{S_{t}} \cM @>v_{s+s',t}>> \cM \otimes_{S_{t}} \cM \otimes_{R_{s+s'}} \cM\\
@VVV @VVV\\
\cM \otimes_{R_{s}} \cM \otimes_{R_{s'}} \cM \otimes_{S_{t}} \cM @>(v_{s,t} \otimes I)(I \otimes v_{s',t})>> \cM \otimes_{S_{t}} \cM \otimes_{R_{s}} \cM \otimes_{R_{s'}} \cM
\end{CD}
\ees
and
\bes
\begin{CD}
\cM \otimes_{R_{s}} \cM \otimes_{S_{t+t'}} \cM @>v_{s,t+t'}>> \cM \otimes_{S_{t+t'}} \cM \otimes_{R_{s}} \cM\\
@VVV @VVV\\
\cM \otimes_{R_{s}} \cM \otimes_{S_{t}} \cM \otimes_{S_{t'}} \cM @>(I \otimes v_{s,t'})(v_{s,t} \otimes I)>> \cM \otimes_{S_{t}} \cM \otimes_{S_{t'}} \cM \otimes_{R_{s}} \cM 
\end{CD} 
\ees
where the vertical maps are the isometries from the above discussion.
\end{defin}
Thus, $R$ and $S$ strongly commute if for all $s,t\geq 0$ $R_s$ and $S_t$ commute strongly, and if this strong commutativty is implemented in a way that is consistent with the semigroup structures. We note that in \cite{ShalitCP0Dil} we gave a weaker definition of strongly commuting semigroups, but that definition was too weak. See \cite{OrrCorr}.

We now give some classes of examples of pairs of strongly commuting CP-semigroups.

\begin{proposition}\label{prop:SCendo}
Let $\alpha = \{\alpha_t\}_{t\geq 0}$, $\beta = \{\beta_t\}_{t\geq 0}$ be two E-semigroups acting on $\cM$, and assume that for all $s,t \geq 0$, $\alpha_s \circ \beta_t = \beta_t \circ \alpha_s$. Then $\alpha$ and $\beta$ strongly commute.
\end{proposition}
\begin{proof}
We construct a family $\{v_{s,t}\}_{(s,t)\in\Rpt}$ as required by Definition \ref{def:SCSG}.
Note that in $\cM \otimes_{\alpha_s} \cM \otimes_{\beta_t} \cM$ we have the equality $a \otimes_{\alpha_s} b \otimes_{\beta_t} c = a \otimes_{\alpha_s} I \otimes_{\beta_t} \beta_t(b) c$. Thus, there is an isomorphism $v_{s,t} : \cM \otimes_{\alpha_s} \cM \otimes_{\beta_t} \cM \rightarrow \cM \otimes_{\beta_t} \cM \otimes_{\alpha_s} \cM$ completely determined by the mapping
\bes
a \otimes_{\alpha_s} I \otimes_{\beta_t} c \mapsto a \otimes_{\beta_t} I \otimes_{\alpha_s} c.
\ees
Clearly, $v_{s,t}(I \otimes I \otimes I) = I \otimes I \otimes I$. We have yet to show that the family $\{v_{s,t}\}_{(s,t)\in\Rpt}$ participates in commutative diagrams as in Definition \ref{def:SCSG}. Denote by $V$ the isometry $\cM \otimes_{\alpha_{s_1+s_2}} \cM \otimes_{\beta_t} \cM \rightarrow \cM \otimes_{\alpha_{s_1}} \cM \otimes_{\alpha_{s_2}} \cM \otimes_{\beta_t} \cM$ sending $a \otimes b \otimes c$ to $a \otimes I \otimes b \otimes c$, and denote by $W$ the similar isometry $\cM \otimes_{\beta_t} \cM \otimes_{\alpha_{s_1+s_2}} \cM \rightarrow \cM \otimes_{\beta_t} \cM \otimes_{\alpha_{s_1}} \cM \otimes_{\alpha_{s_2}} \cM$. For all $a,c \in \cM$, we have
\begin{align*}
W (v_{s_1+s_2,t} (a \otimes_{\alpha_{s_1+s_2}} I \otimes_{\beta_t} c)) &= W (a \otimes_{\beta_t} I \otimes_{\alpha_{s_1+s_2}} c) \\
&= a \otimes_{\beta_t} I \otimes_{\alpha_{s_1}} I \otimes_{\alpha_{s_2}} c,
\end{align*}
while, on the other hand,
\begin{align*}
(v_{s_1,t} \otimes I)(I \otimes v_{s_2,t}) (V (a \otimes_{\alpha_{s_1+s_2}} I \otimes_{\beta_t} c ))
&= (v_{s_1,t} \otimes I)(I \otimes v_{s_2,t}) (a \otimes_{\alpha_{s_1}} I \otimes_{\alpha_{s_2}} I \otimes_{\beta_t} c) \\
&= (v_{s_1,t} \otimes I) (a \otimes_{\alpha_{s_1}} I \otimes_{\beta_t} I \otimes_{\alpha_{s_2}} c) \\
&= a \otimes_{\beta_t} I \otimes_{\alpha_{s_1}} I \otimes_{\alpha_{s_2}} c .
\end{align*}
That establishes one commutative diagram. The other is similar.
\end{proof}

The importance of the above proposition is that it ensures that if $\{\alpha_t\}_{t\geq0}$ and $\{\beta_t\}_{t\geq0}$ are an E-dilation of $\{R_t\}_{t\geq0}$ and $\{S_t\}_{t\geq0}$ (a pair of strongly commuting CP-semigroups), then $\alpha$ and $\beta$ also commute strongly.

\begin{proposition}\label{prop:autCP}
Let $\alpha = \{\alpha_t\}_{t\geq 0}$ be a semigroup of normal $*$-automorphisms on $\cM$, and let $\theta = \{\theta_t\}_{t\geq 0}$ be a CP-semigroup on $\cM$, and assume that for all $s,t \geq 0$, $\alpha_s \circ \theta_t = \theta_t \circ \alpha_s$. Then $\alpha$ and $\theta$ strongly commute.
\end{proposition}
\begin{proof}
The proof is similar to the proof of Proposition \ref{prop:SCendo}, and is omitted.
\end{proof}

\begin{proposition}
Let $T^1 = \{T^1_t\}_{t\geq 0}$ and $T^2 = \{T^2_t\}_{t\geq 0}$ be two commuting semigroups of contractions in $B(H)$, which are continuous in the strong operator topology. For $i=1,2$, let $\phi^i$ be the CP-semigroup defined by
\bes
\phi^i_t(a) = T^i_t a {T^i_t}^* \,\, , \,\, a \in B(H).
\ees
Then $\phi^1$ and $\phi^2$ strongly commute.
\end{proposition}
\begin{proof}
This follows easily from Theorem \ref{thm:Solel}.
\end{proof}

\section{Representing strongly commuting CP-semigroups via product system representations}\label{sec:repSC}

Let us now recall briefly the constructions of \cite{ShalitCP0Dil} which are the central tool in our approach to dilations. This approach is due to Muhly and Solel \cite{MS02}. First we present the underlying strategy.

\subsection{The strategy}\label{subsec:strategy}
Let $\Theta$ be a CP-semigroup over a (cancellative, abelian, unital) semigroup $\cS$,
acting on a von Neumann algebra $\cM$ of operators in $B(H)$. The
dilation is carried out in two main steps. In the first step, a
product system of $\cM'$-correspondences $X$ over $\cS$
is constructed, together with a representation $(\sigma,T)$
of $X$ on $H$ with $\sigma$ being the identity representation, such that for all $a \in \cM, s \in \cS$,
\be\label{eq:rep_rep}
\Theta_s (a) = \widetilde{T_s} \left(I_{X(s)}
\otimes a \right) \widetilde{T_s}^* ,
\ee
where $T_s$ is the
restriction of $T$ to $X(s)$.
\begin{lemma}\label{lem:semigroup}
Let $W$ be completely contractive covariant representation of $X$ on a Hilbert space $G$, such that $W_0$ is unital. Then the family of maps
\bes
\Theta_s : a \mapsto \widetilde{W}_s (I_{X(s)} \otimes a) \widetilde{W}_s^* \,\, , \,\, a \in W_0 (N)',
\ees
is a semigroup of CP maps (indexed by $\cS$) on $W_0 (N)'$. Moreover, if $W$ is an isometric representation, then $\Theta_s$ is a $*$-endomorphism for all $s\in\cS$.
\end{lemma}
\begin{proof}
\cite[Lemma 6.1]{ShalitCP0Dil}
\end{proof}

Having this Lemma in mind,
one sees that a natural way to continue the process of dilation
will be to somehow ``dilate" $(\sigma, T)$ to an isometric representation. In previous works, such as \cite{MS02}, \cite{S06a} and \cite{ShalitCP0Dil}, the construction of a minimal isometric dilation $(\rho,V)$ (representing $X$ on a Hilbert space $K \supseteq H$) of the
representation $(\sigma, T)$ appearing in equation
(\ref{eq:rep_rep}) is the second step of the dilation
process. Then it is shown that if $\cR = \rho(\cM')'$, and
$\alpha$ is defined by
\bes
\alpha_s (a) := \widetilde{V_s}
\left(I_{X(s)} \otimes a \right) \widetilde{V_s}^* \,\,,\,\, a \in \cR,
\ees
then the quadruple $(K,u,\cR,\alpha)$ is an E-dilation for
$(\Theta, \cM)$. (In this context it is interesting to note Skeide's paper \cite{Skeide08}, where this strategy was reversed to prove the existence of isometric dilation for a representation of a product system over $\mb{R}_+$, by dilating the associated one-parameter CP-semigroup).

In the remainder of this paper we shall try to follow the steps mentioned above, but the difficulties that stem for the fact that our semigroup is not-necessarily unit preserving will force us to follow a close but different route, which will unravel as we proceed.

\subsection{Construction of the representation}

Let $\cM$ be a von
Neumann algebra acting on a Hilbert space $H$, let $\{R_t\}_{t\geq0}$ and
$\{S_t\}_{t\geq0}$ be two strongly commuting CP-semigroups on
$\cM$, and $P_{(s,t)} := R_s S_t$. This notation will be fixed for the rest of the paper.
The following is a summary of Section 4.3 in \cite{ShalitCP0Dil}.

In \cite{MS02}, Muhly and Solel associate with every CP-semigroup a product system (of
$W^*$-correspondences over $\cM '$) and a product system representation which represents it as in (\ref{eq:rep_rep}). Let $\{E(t)\}_{t\geq0}$,
$\{F(t)\}_{t\geq0}$ denote the product systems associated with
$\{R_t\}_{t\geq0}$ and $\{S_t\}_{t\geq0}$, respectively, and let
$T^E$, $T^F$ be the corresponding representations. For $s,t \geq 0$, we
denote by $\theta_{s,t}^E$ and $\theta_{s,t}^F$ the unitaries
$$\theta_{s,t}^E : E(s)\otimes_{\cM'}E(t) \rightarrow E(s+t) ,$$
and
$$\theta_{s,t}^F : F(s)\otimes_{\cM'}F(t) \rightarrow F(s+t) .$$

For all $s,t \geq 0$ one can construct an isomorphism of
$W^*$-correspondences
\be
\varphi_{s,t}:E(s) \otimes_{\cM'} F(t)
\rightarrow F(t) \otimes_{\cM'} E(s),
\ee
which are compatible with the semigroup structure as in Definition \ref{def:SCSG}.
The isomorphisms
$\{\varphi_{s,t}\}_{s,t\geq0}$, together with the identity
representations $T^E$, $T^F$, satisfy
the ``commutation" relation:
\be\label{eq:commutation_relation}
\tT_s^E (I_{E(s)} \otimes \tT_t^F) = \tT_t^F (I_{F(t)} \otimes
\tT_s^E) \circ (\varphi_{s,t} \otimes I_H) \quad , t,s \geq0 .
\ee
We define the product system $X$ by
$$X(s,t) := E(s) \otimes F(t) ,$$
and
$$\theta_{(s,t),(s',t')}: X(s,t) \otimes X(s',t') \rightarrow X(s+s',t+t') ,$$
by
$$\theta_{(s,t),(s',t')} =  (\theta_{s,s'}^E \otimes \theta_{t,t'}^F)\circ (I \otimes \varphi_{s',t}^{-1} \otimes I) .$$
For $s,t
\geq 0$, $\xi \in E(s)$ and $\eta \in F(t)$, we define a representation $T$ of $X$ by
$$T_{(s,t)}(\xi \otimes \eta) := T_s^E (\xi) T_t^F (\eta) .$$
A long and technical proof shows that all of our constructions and assertions above our make sense and are correct, and we have the following theorem.
\begin{theorem}\label{thm:rep_SC}(\cite[Theorem 4.12]{ShalitCP0Dil}. See also \cite{OrrCorr})
There exists a two-parameter product system of
$\cM'$-correspondences $X$, and a completely contractive,
covariant representation $T$ of $X$ into
$B(H)$, such that for all $(s,t) \in \Rpt$ and all $a \in \cM$,
the following identity holds: \be\label{eq:rep_SC}
\tT_{(s,t)}(I_{X(s,t)} \otimes a)\tT_{(s,t)}^* = P_{(s,t)}(a) .
\ee
Furthermore, if $P$ is unital, then $T$ is fully coisometric.
\end{theorem}

\begin{remark}
\emph{
It is important to note that the above construction depends on the condition that $R$ and $S$ commute strongly. This condition is used for the existence of the maps $\varphi_{s,t}$ (satisfying Definition \ref{def:SCSG}) and in the proof that $X$ is a product system.}
\end{remark}

\section{Isometric dilation of a product system representation over $\diadp$}\label{sec:iso_dil}

If everything was going according to the plans sketched at the beginning of Subsection \ref{subsec:strategy}, then the title of this section should have been ``Isometric dilation of a product system representation over $\Rpt$".
However, we have finally reached the point where this work starts to differ significantly from \cite{ShalitCP0Dil}. In \cite{ShalitCP0Dil}, since the CP-semigroup was unital, the product system representation which had to be dilated was \emph{fully coisometric}, and this property is crucial for the proof given there for the existence of an isometric dilation. Fully-coisometric product system representations are analogous in a way to semigroups of \emph{coisometries} on a Hilbert space. In the context of classical dilation theory of contraction semigroups on a Hilbert space \cite{SzNF70}, the problem of finding an isometric dilation to a semigroup of coisometries is relatively easy (see also \cite{Douglas}). Here we will only be able to construct an isometric dilation for a product system representation over the set $\diadp$ of positive dyadic pairs, where
$$\diad := \left\{\left(\frac{k}{2^{n}},\frac{m}{2^{n}}\right) : (k,m) \in \mathbb{Z}^2, n \in \mathbb{N} \right\} $$
is the set of dyadic fractions.
This will be sufficient to lead us to our present goal of dilating a CP-semigroup over $\Rpt$, due to some remarkable continuity properties of CP-semigroups, discussed in the next section.

Let $\cM$ be a von Neumann algebra, let $X$ be a product system of $\cM$-correspondences over $\diadp$, and let $H$ be a Hilbert space. Assume that $\sigma$ is a normal representation of $\cM$ on $H$. We denote by $\cH_0$ the space of all finitely supported functions $f$ on $\diadp$ such that $f(n) \in X(n) \otimes_\sigma H$, for all $n \in \diadp$. For any $n = (n_1, n_2) \in \diad$, we denote by $n_+$ the element in $\diadp$ having $\max\{n_i,0 \}$ in its $i$-th entry, and we denote $n_- = n_+ - n$.
\begin{defin}\label{def:positive}
Let $\Phi$ be a function on $\diad$ such that $\Phi(n) \in B(X(n_+) \otimes_\sigma H, X(n_-) \otimes_\sigma H)$, $n \in \diad$. We say that $\Phi$ is \emph{positive definite} if $\Phi (0) = I_{M \otimes_\sigma H}$ and if
\begin{enumerate}
\item \label{it:positive} For all $h \in \cH_0$ we have
\bes
\sum_{m,n \in \diadp} \left\langle \left[ I_{X(m-(m-n)_+)} \otimes \Phi(m-n)\right ] h(m), h(n) \right\rangle \geq 0 .
\ees
\item \label{it:self-adjoint}$\Phi(n) = \Phi(-n)^*$, for all $n \in \diadp$.
\item \label{it:covariance} For all $n,\in \diadp$, $a \in \cM$,
\bes
 \Phi(n)\left(a \otimes I_H \right) = \left(a \otimes I_H \right) \Phi(n) .
\ees
\end{enumerate}
\end{defin}
In item (\ref{it:covariance}) above, $(a \otimes I_H) \xi$ should be interpreted as $\sigma (a) \xi$ if $\xi \in H$, $\varphi_{X(m)}(a) x \otimes h$ if $\xi = x \otimes h$, $h \in H, x \in X(n)$, for some $n \neq 0$, and as $ab \otimes h$ if $\xi = b \otimes h$ for $b \in \cM, h \in H$. This will remain our convention throughout.

We note that the proof of Theorem 3.5 in \cite{S06b} implies the following fact: \emph{if $\Phi$ is a positive definite function on $X$ as above, then there exists a covariant isometric representation $V$ of $X$ on some Hilbert space $K \supseteq H$, such that $H$ is a reducing subspace for $V_0$ with $V_0(\cdot) \big|_H = \sigma(\cdot)$, and such that}
\be\label{eq:dil}
P_{X(n_-) \otimes H} V(n) \big|_{X(n_+) \otimes H} = \Phi(n) \,\, , \,\, n \in \diad,
\ee
\emph{where $V(n) := \widetilde{V}_{n_-}^*\widetilde{V}_{n_+}$}. This fact is the basis of the proof of the following theorem, so we point out that the definition of $V$ in the above mentioned proof has to be modified in an obvious manner and that straightforward calculations (some almost identical and some different from what appeared in the proof) show that $V$ has all the required properties. The main difference is that one has to show that $V$ has the ``semigroup" property.

\begin{theorem}\label{thm:IsoDilDiad}
Let $\cM$ be a $C^*$-algebra, and let $X=\{X(s,t)\}_{(s,t) \in \diadp}$ be a product system of $\cM$-correspondences. Let $T$ be a representation of $X$ on a Hilbert space $H$, with $\sigma = T(0,0)$.
Assume that for all $(s,t) \in \diadp$, the Hilbert space $X(s,t) \otimes_{\sigma} H$ is separable. 
Then there exists an isometric representation $V$ of $X$ on Hilbert space $K$ containing $H$ such that:
\begin{enumerate}
    \item $P_H$ commutes with $V_{(0,0)} (\cM)$, and $V_{(0,0)}(a) P_H = T_{(0,0)}(a) P_H$, for all $a \in \cM$.
    \item\label{it:dilation} $P_H V_{(s,t)}(x)\big|_H = T_{(s,t)}(x)$ for all $(s,t) \in \diadp$, $x \in X(s,t)$.
    \item\label{it:min} $K = \bigvee \{V(x)h : x \in X, h \in H \}$.
    \item\label{it:V*} $P_H V_{(s,t)}(x)\big|_{K \ominus H} = 0$ for all $(s,t) \in \diadp$, $x \in X(s,t)$.
\end{enumerate}
If $\cM$ is a $W^*$-algebra, $X$ a product system of $W^*$-correspondences and $T_0$ is normal, then
$V_0$ is also normal, that is, $V$ is a representation of $W^*$-product systems.
\end{theorem}
A dilation satisfying item (\ref{it:min}) above is called a \emph{minimal} dilation.

\begin{proof}
For any $n \in \mathbb{N}$, the triple $(\sigma, T(2^{-n},0), T(0,2^{-n}))$ defines a c.c. representation of the sub-product system $X^{(n)} = \{X(m/{2^n},k/{2^n})\}_{m,k}$. We will denote $X^{(n)}(m,k) = X(m/{2^n},k/{2^n})$. By Theorem 4.4 in \cite{S06a}, this representation has a covariant isometric dilation $(\rho_n,V_{1,n},V_{2,n})$ on some Hilbert space which we need not refer to. As $n$ increases we get isometric dilations to the restriction of $T$ to fatter and fatter sub-product systems, but the problem is that we do not know exactly how (and if) they sit one inside the other. Our immediate goal is to define a positive definite function $\Phi$ on $\diad$. The heart of the following idea is taken from Ptak's paper \cite{Ptak}.

First we define, for all $n \in \mathbb{N}, (s,t) \in \diad$ an operator $a_n (s,t)$ in $B(X(s_+,t_+) \otimes H, X(s_-,t_-) \otimes H)$. This is done in the following manner. Fixing $(s,t) \in \diad$, there is some $n_{s,t} \in \mathbb{N}$ such that for all $n \geq n_{s,t}$ there are two integers $m_{s,n}$ and $k_{t,n}$ satisfying $(s,t) = (m_{s,n}\cdot 2^{-n},k_{t,n}\cdot 2^{-n})$, and such that $n_{s,t}$ is the minimal natural number with this property. For $n < n_{s,t}$ we define $a_n(s,t) = 0$. For $n \geq n_{s,t}$ we define
\bes
a_n(s,t) = P_{X(s_-,t_-)\otimes H}\overline{V}_n (m_{s,n},k_{t,n}) \big|_{X(s_+,t_+)\otimes H}
\ees
where $\overline{V}_n (m,k):= \left(I_{X^{(n)}(0,k_-)} \otimes (\widetilde{V}_{1,n}^{m_-})^*\right) (\widetilde{V}_{2,n}^{k_-})^* \widetilde{V}_{1,n}^{m_+}(I_{X^{(n)}(m_+,0)} \otimes \widetilde{V}_{2,n}^{k_+})$ (to be precise, one should multiply the right hand side by $U_{(0,k_-),(m_-,0)}\otimes I_H$ on the left and $U_{(m_+,0),(0,k_+)}^{-1} \otimes I_H$ on the right, where $U_{\cdot,\cdot}$ are the multiplication maps of $X^{(n)}$).

For fixed $(s,t) \in \diadp$, and for large enough $n$, we have
$$a_n(s,t) = T(s,t):= \widetilde{T}^*_{(s_-,t_-)} \widetilde{T}_{(s_+,t_+)} = \widetilde{T}_{(s,t)} .$$
Fixing $(s,t) \in \diad$, we have  for all large enough $n$
\begin{align*}
a_n(-s,-t)^*
&=\left(P_{X(s_+,t_+) \otimes H} \left(I_{0,k_+} \otimes (\widetilde{V}_{1,n}^{m_+})^*\right) (\widetilde{V}_{2,n}^{k_+})^* \widetilde{V}_{1,n}^{m_-}(I_{m_-,0} \otimes \widetilde{V}_{2,n}^{k_-}) \right)^* \big|_{X(s_-,t_-)\otimes H} \\
&= P_{X(s_-,t_-) \otimes H} \left(I_{m_-,0} \otimes (\widetilde{V}_{2,n}^{k_-})^*\right) (\widetilde{V}_{1,n}^{m_-})^* \widetilde{V}_{2,n}^{k_+}(I_{0,k_+} \otimes \widetilde{V}_{1,n}^{m_+})\big|_{X(s_+,t_+)\otimes H} \\
(*)&=a_n(s,t)
\end{align*}
where we used the shorthand notations $I_{p,q} = I_{X^{(n)}(p,q)}$, $m=m_{s,n},k=k_{t,n}$, and the equality in marked by (*) is true up to multiplication by the product system multiplication maps $U_{\cdot,\cdot}$. Also, it follows immediately from the covariance properties of $(\rho_n,V_{1,n},V_{2,n})$ that $a_n(s,t)$ intertwines the various interpretations of $(a \otimes I_H), a \in \cM$.

Now that $a_n(s,t)$ is defined, we construct a positive definite function $\Phi$ on $\diad$. For every $(s,t) \in \diad$, $\{a_n(s,t)\}_n$
is a sequence of operators in $B(X(s_+,t_+) \otimes H, X(s_-,t_-) \otimes H)$ with norm less than or equal $1$.
As the unit ball of $B(X(s_+,t_+) \otimes H, X(s_-,t_-) \otimes H)$ is weak operator compact, there is a subsequence
$\{n_k\}_{k=1}^\infty$ of $\mathbb{N}$ such that $a_{n_k}(s,t)$ converges in the weak operator topology (our separability assumptions imply that the unit ball of $B(X(s_+,t_+) \otimes H, X(s_-,t_-) \otimes H)$ is metrizable, hence sequentially compact). In fact, since $\diad$ is countable,
a standard diagonalization procedure will produce $\{n_k\}_{k=1}^\infty$ of $\mathbb{N}$ such that $a_{n_k}(s,t)$ converges weakly for \emph{all} $(s,t) \in \diad$. We define
$$\Phi(s,t) = \textrm{wot --}\lim_{k \rightarrow \infty} a_{n_k}(s,t).$$

By the properties that $a_n(s,t)$ possesses, it follows that for $(s,t) \in \diadp$,
$$\Phi(s,t) = T(s,t).$$
Also, $\Phi$ satisfies items (\ref{it:self-adjoint}) and (\ref{it:covariance}) of Definition \ref{def:positive}. For example, for (\ref{it:covariance}) it is enough to check
\begin{align*}
\langle \Phi(s,t) (a \otimes I) e_i, f_j \rangle
&= \lim_{k \rightarrow \infty }\langle a_{n_k}(s,t) (a \otimes I) e_i, f_j \rangle \\
&= \lim_{k \rightarrow \infty }\langle a_{n_k}(s,t)  e_i, (a \otimes I)^* f_j \rangle \\
&= \langle (a \otimes I) \Phi(s,t)  e_i, f_j \rangle .
\end{align*}
(\ref{it:self-adjoint}) follows similarly. Let us prove that it also satisfies item (\ref{it:positive}). Let $h \in \cH_0$, and consider the sum
\be\label{eq:sum}
\sum_{m,n \in \diadp} \left\langle \left[I_{X(m-(m-n)_+)} \otimes \Phi(m-n)\right] h(m), h(n) \right\rangle .
\ee
We are going to show that this sum is greater than $-\epsilon$, for any $\epsilon >0$. The sum in (\ref{eq:sum}) contains only a finite number, say $N$, of non-zero summands. We may take $k$ large enough to satisfy
\bes
\left|\left\langle \left[I_{m-(m-n)_+} \otimes \Phi(m-n)\right] h(m), h(n) \right\rangle - \left\langle \left[I_{m-(m-n)_+} \otimes a_{n_k}(m-n)\right] h(m), h(n) \right\rangle\right| < \frac{\epsilon}{N} ,
\ees
for all $m,n \in \diadp$. If needed, we take $k$ even larger, so that
\bes
a_{n_k}(d) = P_{X(d_-)\otimes H}\overline{V}_{n_k} (m_{d_1,n},k_{d_2,n}) \big|_{X(d_+)\otimes H}
\ees
for all $d = (d_1,d_2) \in \diad$ that appears in a non-zero inner product in (\ref{eq:sum}). In other words, we assume that all dyads appearing non-trivially in (\ref{eq:sum}) have the form $(p\cdot2^{-n_k},q\cdot2^{-n_k})$, $p,q \in \mathbb{Z}$. But then
\begin{align*}
& \sum_{m,n \in \diadp} \langle I_{m-(m-n)_+} \otimes a_{n_k}(m-n) h(m), h(n) \rangle \\
&= \sum_{m,n \in \diadp} \langle I_{m-(m-n)_+} \otimes P_{X((m-n)_-)\otimes H} \widetilde{U}_{(m-n)_-}^* \widetilde{U}_{(m-n)_+} \big|_{X(m-n)_+} h(m), h(n) \rangle \\
(*)&= \sum_{m,n \in \diadp} \langle I_{m-(m-n)_+} \otimes \widetilde{U}_{(m-n)_-}^* \widetilde{U}_{(m-n)_+} h(m), h(n) \rangle \\
(**)&= \sum_{m,n \in \diadp} \langle \widetilde{U}_{n}^* \widetilde{U}_{m} h(m), h(n) \rangle \\
&= \sum_{m,n \in \diadp} \langle  \widetilde{U}_{m} h(m),\widetilde{U}_{n} h(n) \rangle \geq 0.
\end{align*}
The equality marked by (*) follows from identifying $X(d) \otimes H$ with a subspace of $X(d) \otimes G$, where $G$ is the dilation Hilbert space associated with $U$, and $U$ is the isometric dilation of the restriction of $T$ to $X^{(n_k)}$. The equality marked by (**) follows from
\begin{align*}
\widetilde{U}_{n}^* \widetilde{U}_{m}
&= (I_{n-(n-m)_+} \otimes \widetilde{U}^*_{(n-m)_+})\widetilde{U}^*_{n-(n-m)_+}\widetilde{U}_{m-(m-n)_+} (I_{m-(m-n)_+} \otimes
\widetilde{U}_{(m-n)_+})\\
&= I_{m-(m-n)_+} \otimes \widetilde{U}^*_{(m-n)_-} \widetilde{U}_{(m-n)_+},
\end{align*}
because $n-(n-m)_+ = m-(m-n)_+$ and $(n-m)_+ = (m-n)_-$.
Thus
\bes
\sum_{m,n \in \diadp} \left\langle \left[I_{X(m-(m-n)_+)} \otimes \Phi(m-n)\right] h(m), h(n) \right\rangle \geq -\epsilon
\ees
for all $\epsilon$, thus $\sum_{m,n \in \diadp} \left\langle \left[I_{X(m-(m-n)_+)} \otimes \Phi(m-n)\right] h(m), h(n) \right\rangle \geq 0$, for all $h \in \cH_0$.

We have shown that $\Phi$ is a positive definite function on $\diad$. It follows from the remarks before the statement of the theorem that there is a covariant isometric representation $V$ of $X$ on a Hilbert space $K \supseteq H$ satisfying items (1) and (2) in the statement of the theorem. Given an isometric dilation $V$ on $K$, it is
easy to see that one can cut off part of the space $K$ to obtain a minimal dilation. If we take $V$ to be a minimal dilation, (\ref{it:V*}) follows as well (this point is already discussed in various places for rather general semigroups, e.g. \cite{ShalitReprep}).

The proof of the final assertion of the theorem is not different from the proof given in the proof of the analogous
part in \cite[Theorem 5.2]{ShalitCP0Dil}, and we do not wish to repeat it here.
\end{proof}

\section{Extension of densely parameterized positive semigroups}\label{sec:extension}
Recall our notation: $R$ and $S$ are strongly
commuting CP-semigroups on $\cM \subseteq B(H)$, and $P = \{P_{(s,t)} \}_{(s,t)\in\Rpt}$ is given
by $P_{(s,t)} = R_s\circ S_t$. Using the results of the two sections preceding the current one, we will show in the
next one how an E-dilation
$\{\alpha_d\}_{d\in\diadp}$ can be constructed for the subsemigroup
$\{P_{d}\}_{d\in\diadp}$. In this section we prove the result that will allow us to extend
continuously $\{\alpha_d\}_{d\in\diadp}$ to a semigroup over $\Rpt$ (which will be an E-dilation for $P$).
We follow the ideas of SeLegue, who in \cite[pages 37-38]{SeLegue} proved that a semigroup of unital, normal
$*$-endomorphisms over the positive dyads, which is known to be weakly continuous only on a strong operator
dense subalgebra of $B(H)$, can be extended continuously to an E$_0$-semigroup (over $\Rp$). The crucial
step in SeLegue's argument was to use a result of Arveson \cite[Proposition 1.6]{Arv97} regarding convergence
of nets of normal states on $B(H)$. As we are interested in non-unital semigroups, we will have to generalize
a bit Arveson's result. The proof, however, remains very much the same.

\begin{lemma}\label{lem:arveson}(Arveson \cite[Proposition 1.6]{Arv97})
Let $\cM$ be a direct sum of type $I$ factors, let $\{\rho\}_{i}$ be a net of positive linear functionals, and let $\omega$ be a positive normal linear functional such that
$$\lim_{i} \rho_i(x) = \omega(x)$$
for all compact $x \in \cM$, and also
$$\lim_{i} \rho_i(1) = \omega(1) .$$
Then
$$\lim_{i} \|\rho_i - \omega\| = 0.$$
\end{lemma}
\begin{proof}
We shall show that if $i$ is large enough then $ \|\rho_i - \omega\|$ is arbitrarily small. Let $\epsilon > 0$. Since $\omega$ is normal, there exists a finite rank projection $p \in \cM$ such that
\be\label{eq:1}
\omega(1-p) \leq \epsilon.
\ee
Since $p\cM p$ is a von Neumann algebra on the finite dimensional space $pH$, and since $pxp$ is compact for all $x \in \cM$, we have that
\be\label{eq:2}
\lim_i \sup_{x \in \cM_1}| \rho_i(pxp) - \omega(pxp)| = 0,
\ee
where $\cM_1$ denotes the unit ball of $\cM$.
Now,
\begin{align*}
\|\rho_i - \omega\| &= \sup_{x\in\cM_1} |\rho_i(x) - \omega(x)| \\
&\leq \sup_{x \in \cM_1}|\rho_i(pxp) - \omega(pxp)| + \sup_{x\in \cM_1}|\rho_i(x-pxp)|
+ \sup_{x\in \cM_1}|\omega(x-pxp)|.
\end{align*}
By (\ref{eq:2}), the first term in the last expression is smaller than $\epsilon$ when $i$ is large. We now estimate the second and third terms.
Write $x - pxp = (1-p)x + px(1-p)$. Then
$$\sup_{x\in \cM_1}|\mu(x-pxp)| \leq \sup_{x\in \cM_1}|\mu((1-p)x)|+ \sup_{x\in \cM_1}|\mu(px(1-p))|,$$
with $\mu = \rho_i$ or $\mu = \omega$. But by the Schwartz inequality,
$$|\mu((1-p)x)| \leq \mu(1-p)^{1/2}\|x\|$$
and
$$|\mu(px(1-p))| \leq \mu(1-p)^{1/2}\|px\| \leq \mu(1-p)^{1/2}\|x\| .$$
Thus, using (\ref{eq:1}), we obtain the following estimate for the third term:
$$\sup_{x\in \cM_1}|\omega(x-pxp)| \leq 2 \epsilon^{1/2} .$$
Now, $\rho_i(1) \rightarrow \omega(1)$ and $\rho_i(p) \rightarrow \omega(p)$, thus for all $i$ large enough,
$$\rho_i(1-p) \leq \omega(1-p) + \epsilon \leq 2 \epsilon,$$
so
$$\sup_{x\in \cM_1}|\rho_i(x-pxp)| \leq 4 \epsilon^{1/2} .$$
We conclude that for all $i$ large enough, $\|\rho_i - \omega\| \leq 6 \epsilon^{1/2} + \epsilon$. This completes the proof.
\end{proof}

We now give a somewhat generalized version of SeLegue's Theorem discussed above.
\begin{theorem}\label{thm:SeLegue}
(SeLegue, \cite[pp.~ 37-38]{SeLegue})
Let $\cM \subseteq B(H)$ be direct sum of type $I$ factors. Let $\cS$ be a dense subsemigroup of $\Rp$, and let $\phi = \{\phi_s\}_{s\in \cS}$ be a semigroup over $\cS$ acting on $\cM$, such that $\phi_s$ is a normal positive
linear map for all $s\in\cS$. Assume that for all compact $x \in \cM$ and all $\rho \in \cM_*$,
$$\lim_{\cS \ni s \rightarrow 0} \rho (\phi_s(x)) = \rho(x) \quad \text{\rm and} \quad  \lim_{\cS \ni s \rightarrow 0} \rho (\phi_s(1)) = \rho(1).$$
Then $\phi$ can be extended to a semigroup of normal positive linear maps $\hat{\phi} = \{\hat{\phi}_t\}_{t\geq 0}$
such that $\hat{\phi}_s = \phi_s$ for all $s \in \cS$, satisfying the continuity condition
\be\label{eq:weakstarcont}
\lim_{t \rightarrow t_0} \rho (\hat{\phi}_t (a)) = \rho (\hat{\phi}_{t_0} (a)) \,\,\, \text{\rm for all} \,\,\,  a \in \cM, \rho \in \cM_*.
\ee
Moreover, if $\phi$ consists of
contractions/completely positive maps/unital maps/$*$-endomorphisms then so does $\hat{\phi}$.
\end{theorem}
\begin{proof}
As $\phi_s$ is normal for all $s$, there is a contraction semigroup $T = \{T_s\}_{s\in\cS}$ acting on the predual
$\cM_*$ of $\cM$ such that $T_s^* = \phi_s$ for all $s\in \cS$.
The continuity of $\phi$ at $0$ implies that for all compact $a\in \cM$ and all $\rho \in \cM_*$,
$$T_s \rho (a) \rightarrow \rho(a) \quad \text{\rm and} \quad T_s \rho (1) \rightarrow \rho(1)$$
as $\cS \ni s \rightarrow 0$. Let $\rho$ be a normal state in $\cM_*$. For all $s \in \cS$,
the functional $T_s \rho = \rho \circ \phi_s$ is a positive and normal, because $\phi$ is positive and normal.
Applying Lemma \ref{lem:arveson} to the net $\{T_s \rho\}_{s\in\cS /\{0\}}$, we
obtain that
$$\lim_{\cS \ni s \rightarrow 0}\|T_s \rho - \rho\| = 0 .$$
Any $\rho \in \cM_*$ is a linear combination of normal states, thus
$\lim_{\cS \ni s \rightarrow 0}\|T_s \rho - \rho\| = 0$ for all $\rho \in \cM_*$,
and it follows that for all $s_0 \in \cS$, $\rho \in M_*$,
$$\lim_{\cS \ni s \rightarrow s_0}\|T_s \rho - T_{s_0}\rho\| = 0 .$$
In fact, by standard operator-semigroup methods, for every $\rho$ the map $\cS \ni s \mapsto T_s \rho \in M_*$ is uniformly continuous on bounded intervals, thus it may be extended
to a unique uniformly continuous map $\Rp \longrightarrow M_*$. For all $t \in \Rp$ this gives rise to a well defined
contraction $\hat{T}_t$, such that for $s \in \cS$, $\hat{T}_s = T_s$. It is easy to see that $\{\hat{T}\}_{t \geq 0}$ is a semigroup.

Now define $\hat{\phi}_t = \hat{T}^*_t$. Then $\hat{\phi} = \{\hat{\phi}_t\}_{t\geq 0}$ is a semigroup of normal linear maps
extending $\phi$ and satisfying the continuity condition (\ref{eq:weakstarcont}). With (\ref{eq:weakstarcont}) in mind,
the final claim of the theorem is quite clear,
except, perhaps, the part about $*$-endomorphisms. Assume that $\phi$ is a semigroup of $*$-endomorphisms. For $t \in \Rp$,
$a,b \in \cM$, we have
$$\hat{\phi}_t(ab) = \lim_{\cS \ni s \rightarrow t} \phi_s(ab) = \lim_{\cS \ni s \rightarrow t} \phi_s(a)\phi_s(b) ,$$
where convergence is in the weak operator topology. But $\hat{\phi}$ is a CP-semigroup, thus thus for all $x \in \cM$,
$\phi_s(x)$ converges to $\phi_t(x)$ in the \emph{strong} operator topology as $s \rightarrow t$ (see \cite{MarkiewiczShalit}), so
$$\hat{\phi}_t(ab) =  \lim_{\cS \ni s \rightarrow t} \phi_s(a)\phi_s(b) = \hat{\phi}_t(a)\hat{\phi}_t(b) ,$$
because on bounded sets of $\cM$ multiplication is jointly continuous with respect to the strong operator topology.
\end{proof}

\section{E-dilation of a strongly commuting pair of CP-semigroups}\label{sec:Edilation}

In this section we shall prove our main result, Theorem \ref{thm:main}.
But first we recall without proof two useful results from \cite{ShalitCP0Dil}. The first one appeared
as Lemma \ref{lem:semigroup} above. The second is the following.

\begin{proposition}\label{prop:continuity}
(\cite[Proposition 6.5]{ShalitCP0Dil}) Let $\{R_t\}_{t\geq0}$ and
$\{S_t\}_{t\geq0}$ be two CP-semigroups on $\cM \subseteq B(H)$, where $H$ is a separable Hilbert space. Then the
two-parameter CP-semigroup $P$ defined by
$$P_{(s,t)} := R_s S_t$$
is strongly continuous, that is, for all $a \in \cM$, the map $\Rpt \ni (s,t) \mapsto P_{(s,t)}(a)$ is strongly continuous. Moreover, $P$ is \emph{jointly} continuous on $\Rpt \times \cM$, endowed with the standard$\times$strong-operator topology.

\end{proposition}

Now we are ready to prove the main result:
\begin{theorem}\label{thm:main}
Let $\{R_t\}_{t\geq0}$ and $\{S_t\}_{t\geq0}$ be two strongly
commuting CP-semigroups on $B(H)$, where $H$ is a separable Hilbert space. Then the two-parameter CP-semigroup $P$ defined by
$$P_{(s,t)} := R_s S_t$$
has a minimal E-dilation $(K,u,B(K),\alpha)$, where $K$ is separable.
\end{theorem}

\begin{proof}
Let $X$ and $T$ be the product system of Hilbert spaces and the
product system representation given by Theorem \ref{thm:rep_SC}.
We consider the sub-product system $\check{X}=\{X(s)\}_{s\in\diadp}$
of $X$ represented on $H$ by the sub-representation $\check{T}=\{T_s\}_{s\in\diadp}$
of $T$. The proof of \cite[Proposition 4.2]{MS02} shows that $X(t_1,t_2) = E(t_1) \otimes F(t_2)$ is a separable Hilbert space for all $t_1,t_2 \geq 0$, and it follows that for all $s \in \diadp$ the Hilbert space
$X(s) \otimes_{T_0} H$ is separable. By Theorem
\ref{thm:IsoDilDiad}, there is a minimal isometric
dilation ${V} = \{{V}_s\}_{s\in\diadp}$ of $\check{T}$, representing $\check{X}$ on a Hilbert space ${K} \supseteq H$.
We define a semigroup ${\alpha}=\{{\alpha}_s\}_{s\in\diadp}$ on $B({K})$ by
\bes
{\alpha}_s(a) =
\widetilde{{V}}_s(I \otimes a)\widetilde{{V}}_s^* \,\, , \,\, s\in\diadp ,
a\in B({K}).
\ees
By Lemma \ref{lem:semigroup}, ${\alpha}$ is a semigroup of normal
$*$-endomorphisms on $B({K})$. Denote by $p$ the orthogonal projection
of ${K}$ onto $H$. It is clear that $B(H)$ is the corner $B(H) = pB({K})p \subseteq B({K})$.
To see that ${\alpha}$ is a dilation of $\{P_{s}\}_{s \in \diadp}$, we fix $s \in \diadp$ and $b \in B({K})$ and we compute
\begin{align*}
P_s(p b p)
&= \widetilde{T}_s(I \otimes p b p) \widetilde{T}_s^*  \\
(*)&= p \widetilde{{V}}_s(I \otimes p)(I \otimes b)(I \otimes p) \widetilde{{V}}_s^* p  \\
(**)&= p \widetilde{{V}}_s(I \otimes  b) \widetilde{{V}}_s^* p  \\
&= p {\alpha}_s (b) p .
\end{align*}
The equalities marked by (*) and (**) are justified by items \ref{it:dilation} and \ref{it:V*} of Theorem \ref{thm:IsoDilDiad}, respectively. 

Up to this point, the proof has been simple and straightforward. One may guess that our next step is
to show that ${\alpha}$ is continuous and to extend it
to a semigroup over $\Rpt$, using Theorem \ref{thm:SeLegue}. This is true, but carrying out this plan turns out to be
rather delicate.


We define two (``one-parameter") semigroups $\beta =
\{\beta_t\}_{t\in \mathbb{D}_+}$ and $\gamma = \{\gamma_t\}_{t\in
\mathbb{D}_+}$ on $B(K)$ by
\be\label{eq:betagamma}
\beta_t = \alpha_{(t,0)} \quad  \text{\rm and} \quad \gamma_t =
\alpha_{(0,t)} .
\ee
By Proposition \ref{prop:continuity}, if we will be able to extend $\beta$ and $\gamma$
to continuous E-semigroups $\hat{\beta}$ and $\hat{\gamma}$ over $\Rp$, then the semigroup $\hat{\alpha} = \{\hat{\alpha}_{(s,t)}\}_{(s,t)\in \Rpt}$ given by
$$\hat{\alpha}_{(s,t)} = \hat{\beta}_s \circ \hat{\gamma}_t $$
will be the sought after E-dilation of $P$. The rest of the proof is mostly
dedicated to showing that $\beta$ and $\gamma$ can be continuously extended. As we demonstrate the
extendability of $\beta$ and $\gamma$, we show that $(p, B(K), \alpha)$ is a minimal dilation of $(B(H),\{P_s\}_{s\in\cS})$, and this will complete the proof of Theorem \ref{thm:main}.

Because ${V}$ is a minimal dilation of $\check{T}$, we have
$$K := \bigvee_{s \in \diadp}{V}_s (X(s))H .$$
An important observation is this:
\be\label{eq:K}
K = \bigvee \left\{\alpha_{t_1}(m_1)\cdots\alpha_{t_k}(m_k)h:
k \in \mathbb{N}, t_i\in\diadp, m_i\in B(H), h \in H \right\}.
\ee
(When $k=0$, we interpret the product $\alpha_{t_1}(m_1)\cdots\alpha_{t_k}(m_k)h$ as $h$). In Step 2 of the proof of Theorem 6.6 from \cite{ShalitCP0Dil} this equality is proved (in that paper the situation was a little simpler and one did not have to consider products $\alpha_{t_1}(m_1)\cdots\alpha_{t_k}(m_k)h$ with $k=0$. The proof, however, holds in this case as well, as long as one does takes such products). We do not
repeat the arguments given there, because to understand them one must go into the details of
the Muhly-Solel construction (of a product system representation representing a CP-semigroup).

In fact, In Step 2 of the proof of Theorem 6.6 from \cite{ShalitCP0Dil}, a slightly stronger assertion than (\ref{eq:K})
is proved, namely
\be\label{eq:partition1}
K = \bigvee \alpha_{(s_m,t_n)}(B(H)) \alpha_{(s_m,t_{n-1})}(B(H)) \cdots \alpha_{(s_m,t_1)}(B(H)) \alpha_{(s_{m},0)}(B(H))  \cdots \alpha_{(s_1,0)}(B(H)) H
\ee
where in the right hand side of the above expression we run over all pairs $(s,t)\in\diadp$ and all partitions $\{0=s_0 < \ldots < s_m=s\}$ and $\{0=t_0 < \ldots < t_n = t \}$ of $[0,s]$ and $[0,t]$.

Using (\ref{eq:K}), we can show that $(p, B(K), \alpha)$ is a minimal dilation.
Define
\be\label{eq:define R}
\cR =
W^*\left(\bigcup_{s\in\diadp}\alpha_s(B(H)) \right) .
\ee
Note that $K = [\cR H]$.
But the central projection
of $p$ in $\cR$ is the projection on $[\cR p K] = [\cR H] = K$, that is $I_K$. We will now show that $\cR = B(K)$, and this will prove that the central projection of $p$ in $B(K)$ is $I_K$, so by both definitions \ref{def:min_dil} and \ref{def:min_dil_Arv} $(p, B(K), \alpha)$ is a minimal dilation.

To see that $\cR = B(K)$, let $q\in B(K)$ be a projection in $\cR'$. In particular, $pq = qp = pqp$, so $qp$ is a projection
$B(H)$ which commutes with $B(H)$, thus $qp$ is either $0$ or $I_H$.

If it is $0$ then for all $t_i\in\diadp, m_i\in B(H), h \in H$,
$$q \alpha_{t_1}(m_1)\cdots\alpha_{t_k}(m_k)h = \alpha_{t_1}(m_1)\cdots\alpha_{t_k}(m_k) qp h = 0 ,$$
so $qK = 0$ and $q=0$.

If $qp = I_H$ then for all $0<t_i\in\diadp, m_i\in B(H), h \in H$,
\begin{align*}
q \alpha_{t_1}(m_1)\cdots\alpha_{t_k}(m_k)h
&= \alpha_{t_1}(m_1)\cdots\alpha_{t_k}(m_k) qp h \\
&= \alpha_{t_1}(m_1)\cdots\alpha_{t_k}(m_k)h ,
\end{align*}
so $qK = K$ and $q=I_K$. We see that the only projections in $\cR'$ are $0$ and $I_K$, so
$\cR' = \mathbb{C}\cdot I_K$, hence $\cR = \cR'' = B(K)$. This completes the proof of minimality.

We return to showing that $\beta$ and $\gamma$ can be continuously extended to $\Rp$. Let
$$\dpp = \left\{\frac{m}{2^n} : 0<m,n\in \mathbb{N} \right\},$$
put $\diadpp = \dpp \times \dpp$, and define
\be\label{eq:KS}
K_0 = \bigvee \left\{\alpha_{t_1}(m_1)\cdots\alpha_{t_k}(m_k)h:
k \in \mathbb{N}, t_i\in\diadpp, m_i\in B(H), h \in H \right\}.
\ee
We shall use (\ref{eq:partition1}) to prove that $K = K_0$. First, let us show that $H \subseteq K_0$.
Let
$$G_0 = \bigvee_{t\in\diadpp}\alpha_{t}(B(H))H $$
and $G = H \vee G_0$.
For $t \leq s \in \diadpp$, $a \in B(H)$ and $h,g \in H$, we find that
\begin{align*}
\langle \alpha_{t}(p)h,\alpha_{s}(a)g \rangle &=
\langle \alpha_{s}(a^*) \alpha_{t}(p)h, g \rangle \\
&= \langle \alpha_{t}( \alpha_{s-t}(a^*) p)h, g \rangle \\
&= \langle P_{t}( P_{s-t}(pa^* p) p)h, g \rangle \\
&\stackrel{t\rightarrow 0}{\longrightarrow} \langle P_{s}(pa^* p) h, g \rangle \\
&= \langle h,\alpha_{s}(a)g \rangle .
\end{align*}
Similarly,
$$\langle \alpha_{t}(p)h,g \rangle  \stackrel{t\rightarrow 0}{\longrightarrow} \langle h,g \rangle .$$
We see that in $G$, $\alpha_{t}(p)h \rightarrow h$ weakly, thus $H$ is in the weak closure of $G_0$ in $G$.
The weak closure being equal to the strong closure, we have $H \subseteq G_0 \subseteq K_0$.

The set
$$\{\alpha_{s}(b) h : s\in \diadpp, b\in \bigcup_{t\in\diadp}\alpha_t(B(K)), h\in H \}$$
is total in $K_0$. To see this, note that every element of the form
$$\alpha_{t_1}(m_1)\cdots\alpha_{t_k}(m_k)h ,$$
with $t_i\in\diadpp, m_i\in B(H)$ and $h \in H$,
can be written as
$$\alpha_s (\alpha_{t_1-s}(m_1)\cdots\alpha_{t_k-s}(m_k)) h ,$$
where $s \in \diadpp$ is smaller than $t_i$, $i=1, 2, \ldots, k$.

Let $\alpha_{s_1}(a_1)h_1$ and $\alpha_{s_2}(a_2)h_2$
be in $K_0$, where $s_1, s_2 \in \diadpp$,
$a_1,a_2 \in B(K)$ and $h_1,h_2 \in H$. Take $a \in B(H)$ and $t \in
\diadp$ such that $t < s_1, s_2$. Now compute:
\begin{align*}
\langle\alpha_t(a)\alpha_{s_1}(a_1)h_1,\alpha_{s_2}(a_2)h_2\rangle
&= \langle \alpha_{s_2}(a_2^*) \alpha_{t}(a)\alpha_{s_1}(a_1)h_1,h_2\rangle \\
&= \langle \alpha_{t}\left(\alpha_{s_2-t}(a_2^*) a \alpha_{s_1-t}(a_1)\right)h_1,h_2\rangle \\
&= \langle P_{t}\left(p\alpha_{s_2-t}(a_2^*) pap\alpha_{s_1-t}(a_1)\right)h_1,h_2\rangle \\
&= \langle P_{t}\left(P_{s_2-t}(p a_2^* p) a P_{s_1-t}(p a_1 p)\right)h_1,h_2\rangle \\
(*)&\stackrel{t\rightarrow 0}{\longrightarrow} \langle P_{s_2}(p a_2^* p) a P_{s_1}(p a_1 p) h_1,h_2\rangle \\
&= \langle a \alpha_{s_1}(a_1)h_1,\alpha_{s_2}(a_2) h_2\rangle .
\end{align*}
If we let $p_0$ denote the orthogonal projection of $K$ on $K_0$, we see that
$p_0 \alpha_{t}(a) p_0 \rightarrow p_0 a p_0$ in the weak operator topology as $t
\rightarrow 0$, for all $a \in B(H)$ (the convergence in $(*)$ is due to Proposition \ref{prop:continuity}).
Since $H \subseteq K_0$, one has $p \leq p_0$, and $p_0 a p_0 = a$ for all $a \in B(H)$, thus
\be\label{eq:p_0}
\forall a \in B(H). p_0 \alpha_{t}(a) p_0 \rightarrow a \,\, \text{as} \,\, t \rightarrow 0,
\ee
where convergence is in the weak operator topology.

By (\ref{eq:partition1}), $K$ is spanned by elements of the form
\be\label{eq:g}
g = \alpha_{s_1}(\alpha_{s_2}(\cdots (\alpha_{s_m}(a_m) a_{m-1})\cdots )a_1)h ,
\ee
for $a_i \in B(H), s_i \in \diadp, i=1,\ldots m$ and $h\in H$. Let us show how such an element can be
approximated in norm arbitrarily well by elements of the same form with all $s_i$'s in $\diadpp$ (it is clear that if
all $s_i$'s are in $\diadpp$, then this element is in $K_0$).

Assume that we wish to approximate a fixed element
$g$ as in (\ref{eq:g})
by elements
of the from $\alpha_{t_1}(\cdots(\alpha_{t_m}(b_m) b_{m-1}\cdots)b_1)h'$, where $t_m \in \diadpp$.
We consider
$$g_t = \alpha_{s_1}(\alpha_{s_2}(\cdots(\alpha_{s_m+t}(a_m) a_{m-1})\cdots a_2)a_1)h$$
with $t \in \diadpp$, and compute
\begin{align*}
\langle g_t,g_t \rangle
&=\big\langle \alpha_{s_1}(\cdots(\alpha_{s_m+t}(a_m) a_{m-1}\cdots)a_1)h, \alpha_{s_1}(\cdots(\alpha_{s_m+t}(a_m) a_{m-1}\cdots)a_1)h \big\rangle \\
&= \big\langle\alpha_{s_1}(a_1^* \cdots a_{m-1}^* \alpha_{s_m+t}(a_m^*)\cdots) \alpha_{s_1}(\cdots(\alpha_{s_m+t}(a_m) a_{m-1}\cdots )a_1)h, h \big\rangle \\
&= \big\langle\alpha_{s_1}(a_1^* \alpha_{s_2}(a_2^* \cdots a_{m-1}^* \alpha_{s_m+t}(a_m^* a_m)a_{m-1} \cdots a_{2}))a_1)h, h \big\rangle \\
&= \big\langle\alpha_{s_1}(a_1^* \alpha_{s_2}(a_2^* \cdots a_{m-1}^* P_{s_m+t}(p a^*_m a_m p)a_{m-1} \cdots a_{2}))a_1)h, h \big\rangle \\
&\stackrel{t\rightarrow 0}{\longrightarrow}\big\langle\alpha_{s_1}(a_1^* \alpha_{s_2}(a_2^* \cdots a_{m-1}^* P_{s_m}(p a^*_m a_m p)a_{m-1} \cdots a_{2}))a_1)h, h \big\rangle \\
&= \big\langle\alpha_{s_1}(a_1^* \alpha_{s_2}(a_2^* \cdots a_{m-1}^* \alpha_{s_m}(a_m^* a_m) a_{m-1} \cdots a_{2}))a_1)h, h \big\rangle \\
&= \langle g,g \rangle .
\end{align*}
In addition , we have
\begin{align*}
\langle g,g_t \rangle
&=\big\langle \alpha_{s_1}(\cdots \alpha_{s_{m-1}}(\alpha_{s_m}(a_m) a_{m-1}) \cdots )a_1)h, \alpha_{s_1}(\cdots \alpha_{s_{m-1}}(\alpha_{s_m + t}(a_m) a_{m-1}) \cdots )a_1)h \big\rangle \\
&= \big\langle\alpha_{s_1}(a_1^* \cdots \alpha_{s_{m-1}}(a_{m-1}^* \alpha_{s_m+t}(a^*_m))\cdots) \alpha_{s_1}(\cdots \alpha_{s_{m-1}}(\alpha_{s_m}(a_m) a_{m-1}) \cdots )a_1)h, h \big\rangle \\
&= \big\langle\alpha_{s_1}(a_1^* \cdots \alpha_{s_{m-1}}(a_{m-1}^* \alpha_{s_m}(\alpha_{t}(a_m^*) a_m) a_{m-1})\cdots )a_1)h, h \big\rangle.
\end{align*}
But $a_m = p a_m = p_0 p a_m$, and $p_0$ commutes with $\alpha_t(a_m^*)$, thus
$$\alpha_{t}(a^*_m)a_m = p_0 \alpha_{t}(a^*_m)p_0 a_m \rightarrow a^*_m a_m $$
in the weak operator topology as $t \rightarrow 0$, by (\ref{eq:p_0}). Since $\alpha_{s_i}$ is normal for all $i$, we obtain
$$\langle g,g_t \rangle \stackrel{t\rightarrow 0}{\longrightarrow} \langle g, g \rangle .$$
This allows us to conclude that
$$\|g_t - g\|^2 = \langle g_t , g_t \rangle - 2 \Re \langle g_t, g \rangle + \langle g, g \rangle \stackrel{t\rightarrow 0}{\longrightarrow} 0 ,$$
which shows that $K$ is spanned by elements as $g$ with positive $s_m$. An inductive argument, using the same reasoning
and (\ref{eq:p_0}), shows that $K$ is spanned by elements as $g$ with positive $s_i$ for all $i$, thus $K = K_0$.

Since $p_0 = I_K$, equation (\ref{eq:p_0}) now translates to the weak operator convergence
$$\alpha_t(a) \rightarrow a \,\, \text{as} \,\, t \rightarrow 0 ,$$
for all $a\in B(H)$.
For all $k \in K$, we have
\begin{align*}
\|\alpha_t(a)k - a k \|^2 &=
\langle \alpha_t(a)^* \alpha_t(a) k, k\rangle - 2\Re\langle \alpha_t(a) k , ak \rangle + \|ak\|^2 \\
&= \langle \alpha_t(a^* a) k, k\rangle - 2\Re\langle \alpha_t(a) k , ak \rangle + \|ak\|^2 \\
&\stackrel{s\rightarrow 0}{\longrightarrow} 0,
\end{align*}
thus $\alpha_t(a) \rightarrow a$ in the strong operator topology as $t
\rightarrow 0$, for all $a \in B(H)$.
For all $s\in\diadp$, $\alpha_s$ is normal, thus
$\alpha_t(\alpha_s(a)) = \alpha_s(\alpha_t(a)) \rightarrow
\alpha_s(a)$ in the strong operator topology as $\diadp \ni t \rightarrow 0$.
Denote by $\cA$ the $C^*$-algebra generated by $\bigcup_{s\in\diadp} \alpha_s(B(H))$.
Then we conclude that for all $a \in \cA$, both $\beta_t(a)$ and $\gamma_t(a)$ (recall equation (\ref{eq:betagamma})) converge in the strong operator topology
to $a$ as $\dpp \ni t \rightarrow 0$.

As we have seen above, $W^*(\cA) = \cR = B(K)$, that is, $\cA$ is a $C^*$-algebra strong operator dense in $B(K)$. In particular, $\cA$ acts irreducibly on $K$.
Since $\cA \supseteq B(H)$, it contains nonzero compact operators, and now by
\cite[Proposition 10.4.10]{KR}, we conclude that
$\cA$ contains \emph{all} compact operators in $B(K)$. In order to use
Theorem \ref{thm:SeLegue} we must show that $\beta_t(1)$ and $\gamma_t(1)$ converge weakly to $1$ as $\dpp \ni t \rightarrow 0$. By a polarization argument, it is enough to show that that
\be\label{eq:weakconverge}
\langle \alpha_t(1) k, k \rangle \rightarrow \|k\|^2 \,\,\,\, \text{\rm as}  \,\,\,\, \diadp \ni t \rightarrow 0
\ee
for all $k$ in a total subset of $K$. But taking $h \in H$, $b \in B(K)$ and $s \in \diadpp$, we have that
for $t \leq s$,
$$\alpha_t(1) \alpha_s (b) h = \alpha_t(1) \alpha_s (1) \alpha_s (b) h = \alpha_s(b) h ,$$
because $\alpha_t(1) \leq \alpha_s(1)$. (\ref{eq:weakconverge}) follows. 
This means that we may use Theorem \ref{thm:SeLegue} to deduce that $\beta$ and $\gamma$ extend to E-semigroups $\hat{\beta}$ and $\hat{\gamma}$ over $\Rp$. By Proposition \ref{prop:continuity}, $\{\hat{\beta}_s \circ \hat{\gamma}_t\}_{(s,t)\in \Rpt}$ is a two-parameter E-semigroup dilating $P$.

As the statement regarding the separability of $K$ is clear, the proof is now complete.
\end{proof}

Using Proposition \ref{prop:autCP} we immediately obtain the following corollary.
\begin{corollary}\label{cor:CPaut}
Let $\{R_t\}_{t\geq0}$ and $\{S_t\}_{t\geq0}$ be two commuting CP-semigroups on $B(H)$, where $H$ is a separable Hilbert space. Assume that $R_t$ is a $*$-automorphism of $B(H)$ for all $t\geq 0$.Then the two-parameter CP-semigroup $P$ defined by
$$P_{(s,t)} := R_s S_t$$
has a minimal E-dilation $(K,u,B(K),\alpha)$, where $K$ is separable.
\end{corollary}

\section{Acknowledgements}
This research is part of the author's 
PhD. thesis, done under the supervision of Baruch Solel, whose guidance and support are gratefully appreciated. The author also wishes to thank the referee for some helpful remarks. 


\end{document}